\documentclass[14pt]{amsart}

\usepackage{cases}
\usepackage{amsmath}
\usepackage{amsfonts}
\usepackage{bm}
\usepackage{amsfonts,amsmath,amssymb,amscd,bbm,amsthm,mathrsfs,dsfont}
\usepackage{mathrsfs}
\usepackage{pb-diagram}
\usepackage{amssymb}

\newtheorem{Theorem}{Theorem}[section]
\newtheorem{Lemma}[Theorem]{Lemma}
\newtheorem{Definition}[Theorem]{Definition}

\newtheorem{Proposition}[Theorem]{Proposition}

\newtheorem{Remark}[Theorem]{Remark}
\newtheorem{Conjecture}[Theorem]{Conjecture}

\date{version of \today}

 \normalbaselines
\title{A conjecture on $C$-matrices of cluster algebras }

\author{Peigen Cao $\;\;\;\;\;\;$ Min Huang $\;\;\;\;\;\;$ Fang Li $\;\;\;\;\;\;$}
\address{Peigen Cao
\newline Department
of Mathematics, Zhejiang University (Yuquan Campus), Hangzhou, Zhejiang
310027,  P.R.China}
\email{peigencao@126.com}

\address{Min Huang
\newline Department
of Mathematics, Zhejiang University (Yuquan Campus), Hangzhou, Zhejiang
310027,  P.R.China}
\email{minhuang1989@hotmail.com}

\address{Fang Li
\newline Department
of Mathematics, Zhejiang University (Yuquan Campus), Hangzhou, Zhejiang
310027, P.R.China}
\email{fangli@zju.edu.cn}

\begin{document}
\renewcommand{\thefootnote}{\alph{footnote}}

\begin{abstract}

For a skew-symmetrizable cluster algebra $\mathcal A_{t_0}$ with principal coefficients at $t_0$, we prove that each seed $\Sigma_t$ of  $\mathcal A_{t_0}$ is uniquely determined by its {\bf C-matrix}, which was proposed by  Fomin and  Zelevinsky in \cite{FZ3} as a conjecture. Our proof is based on the fact that the positivity of cluster variables and sign-coherence of $c$-vectors hold for $\mathcal A_{t_0}$, which was actually verified in \cite{GHKK}. More  discussion is given in the sign-skew-symmetric case so as to obtain a conclusion as weak version of the conjecture in this general case.
\end{abstract}

\maketitle
\bigskip

\section{Introduction and preliminaries}
Cluster algebras with principal coefficients are important research objects  in the theory of cluster algebra.  {\bf $g$}-vectors and the related {\bf $c$}-vectors are introduced to describe cluster variables and coefficients in some sense. As shown in \cite{NZ}, {\bf $c$}-vectors and {\bf $g$}-vectors are closely related with each other, where the {\bf $c$}-vectors (respectively, {\bf $g$}-vectors) are defined as  the column vectors of {\bf C-matrices} (respectively, {\bf G-matrices}) of a cluster algebra with principal coefficients.
In \cite{FZ3}, the authors conjectured that

\begin{Conjecture} \label{coj1} (Conjecture 4.7 of \cite{FZ3} )
Let $\mathcal A_{t_0}$ be the cluster algebra with principal coefficients at $\Sigma_{t_0}$ (or say at $t_0$), $\Sigma_t=(X_t,\tilde B_t)$ be a seed of $\mathcal A_{t_0}$ obtained from $\Sigma_{t_0}$ by iterated mutations. Then $\Sigma_t$ is uniquely determined by $C_t$, where $C_t$ is the lower part of $\tilde B_t=\begin{pmatrix}B_t\\C_t\end{pmatrix}$.
\end{Conjecture}

Conjecture \ref{coj1} is true in case $\mathcal A_{t_0}$ is of finite type \cite{RE}, and in the skew-symmetric case \cite{IKKN} using the methods in \cite{PP,PP1}.

In this paper, based on some fundamental notions and methods, we give an affirmation to Conjecture \ref{coj1} in the case that $B_{t_0}$ is  skew-symmetrizable.

Let $\mathcal A_{t_0}$  be a skew-symmetrizable cluster algebra with principal coefficients at $t_0$.  In this paper, we give an affirmation to Conjecture \ref{coj1} based on the fact that  the positivity of cluster variables and sign-coherence of $c$-vectors hold for $\mathcal A_{t_0}$, which was actually verified in \cite{GHKK}.

This paper is organized as follows. In this section, some basic definitions are introduced. In Section 2, we give the main result, i.e. Theorem \ref{mainthm} which gives an affirmation to Conjecture \ref{coj1} in skew-symmetrizable case. In Section 3, we discuss the weak version of Conjecture \ref{coj1} for acyclic totally sign-skew-symmetric cluster algebra using the same method in Section 2.
\\

\begin{Definition}
Let $A$ be an $(m+n)\times n$ integer matrix. The mutation of $A$  in direction $k\in\{1,2,\cdots,n\}$ is the $(m+n)\times n$ matrix $\mu_k(A)=(a_{ij}^{\prime})$, satisfying
\begin{eqnarray}\label{bmutation}
a_{ij}^{\prime}=\begin{cases}-a_{ij}~,& i=k\text{ or } j=k;\\ a_{ij}+a_{ik}[-a_{kj}]_++[a_{ik}]_+a_{kj}~,&otherwise,\end{cases}
\end{eqnarray}
 where $[a]_+=max\{a,0\}$ for any $a\in \mathbb R$.
 \end{Definition}

Recall that an $n\times n$ integer matrix $B=(b_{ij})$ is called {\bf sign-skew-symmetric} if $b_{ij}b_{ji}<0$ or $b_{ij}=b_{ji}=0$ for any $i,j=1,2,\cdots,n$.
A sign-skew-symmetric $B$ is {\bf totally sign-skew-symmetric} if any matrix $B^{\prime}$ obtained from $B$ by a sequence of mutations is sign-skew-symmetric.

 An $n\times n$ integer matrix $B=(b_{ij})$ is called  {\bf skew-symmetrizable} if there is a positive integer diagonal matrix $S$ such that $SB$ is skew-symmetric, where $S$ is said to be the {\bf skew-symmetrizer} of $B$.

It is well-known that skew-symmetry, skew-symmetrizability are invariant under mutation, thus skew-symmetrizable integer matrices are always totally sign-skew-symmetric.

 For a sign-skew-symmetric matrix $B$, we can encode the sign pattern of entries of $B$ by the directed graph $\Gamma(B)$ with the vertices $1,2,\cdots,n$ and the directed edges $(i,j)$ for $b_{ij}>0$. A sign-skew-symmetric matrix $B$ is called {\bf acyclic} if $\Gamma(B)$ has no oriented cycles. As shown in \cite{HL}, an acyclic sign-skew-symmetric integer matrix $B$ is always totally sign-skew-symmetric.

Let $\tilde B=\begin{pmatrix}B_{n\times n}\\C_{m\times n}\end{pmatrix}=(\tilde b_{ij})$ be an $(m+n)\times n$ integer matrix  such that $B$ is totally sign-skew-symmetric. $\tilde{B}$ is said to be {\bf skew-symmetric, skew-symmetrizable, (totally) sign-skew-symmetric} respectively if so is $B$.

  \begin{Definition} Let $\mathcal F$ be the field of rational functions in $m+n$ independent variables with coefficients $\mathbb Q$.
A {\bf (labeled) seed} $\Sigma$ in $\mathcal F$ is a pair $(X, \tilde B)$, where

 (i)~  $\tilde B=\begin{pmatrix}B\\C\end{pmatrix}$ is given as above, called an {\bf extended exchange matrix}. The $n\times n$ submatrix $B$ is called an {\bf exchange matrix}, the $m\times n$ submatrix $C$ is called a {\bf coefficients matrix}.

 (ii)~$X=(x_1,\cdots,x_{m+n})$ is a $(m+n)$-tuple, called an {\bf extended cluster} with $x_1,\cdots,x_{m+n}$ forming a free generating set of $\mathcal F$.
 $x_1,\cdots,x_n$ are called {\bf cluster variables} and $x_{n+1},\cdots,x_{m+n}$ are called {\bf frozen variables}. Sometimes we also write frozen variables as $y_1,\cdots,y_m$ in this paper, i.e. $x_{n+i}=y_i$ for $i=1,\cdots,m$.
\end{Definition}

\begin{Definition}
Let $k\in \{1,2,\cdots,n\}$. The {\bf mutation} of the seed $\Sigma=(X,\tilde B)$ in the direction $k$ is the seed $\Sigma^{\prime}=(X^{\prime},\tilde B^{\prime})$ where

(i)~$X^{\prime}=(x_1^{\prime},\cdots,x_{m+n}^{\prime})$ with $x_i^{\prime}=x_i$ if $i\neq k$ and
$x_k^{\prime}x_k=\prod\limits_{i=1}^{m+n} x_i^{[\tilde b_{ik}]_+}+\prod\limits_{i=1}^{m+n} x_i^{[-\tilde b_{ik}]_+}$.

(ii)~$\tilde B^{\prime}=(\tilde b_{ij}^{\prime})$ is defined by
\begin{eqnarray}\label{bmutation}
\tilde b_{ij}^{\prime}=\begin{cases}-\tilde b_{ij}~,& i=k\text{ or } j=k;\\ \tilde b_{ij}+\tilde b_{ik}[-\tilde b_{kj}]_++[\tilde b_{ik}]_+\tilde b_{kj}~,&otherwise.\end{cases}
\end{eqnarray}
We write $\Sigma^{\prime}=\mu_k(\Sigma)$, where $\mu_k$ is called the {\bf mutation map} in $k$.
\end{Definition}
It can be seen that mutation map is always an involution, i.e. $\mu_k\mu_k(\Sigma)=\Sigma$.
\begin{Definition}
The {\bf cluster algebra} $\mathcal A(\Sigma)$ is the $\mathbb Z[y_1,\cdots,y_m]$-subalgebra of $\mathcal F$ generated by all cluster variables obtained from the seed $\Sigma=(X,\tilde B)$ by iterated mutations. $\Sigma$ is called an {\bf initial seed} of $\mathcal A(\Sigma)$.
\end{Definition}

A cluster algebra is called {\bf acyclic} if it admits an acyclic exchange matrix.

Denote by $I_n$ an $n\times n$ identity matrix. If $\tilde B=\begin{pmatrix}B\\I_n\end{pmatrix}$, then $\mathcal A(\Sigma)$ is called the cluster algebra with {\bf principal coefficients at $\Sigma$}. In this case, $m=n$.

\begin{Definition}
(\cite{M})  A {\bf cluster pattern}  $M$ in $\mathcal F$ is an assignment for each seed  $\Sigma_t$ to a vertex $t$ of the $n$-regular tree $\mathbb T_n$, such that for any edge $t^{~\underline{\quad k \quad}}~ t^{\prime},~\Sigma_{t^{\prime}}=\mu_k(\Sigma_t)$. The pair of $\Sigma_t$ are written as $\Sigma_t=(X_t,\tilde B_t)$ with
$X_t=(x_{1;t},\cdots,x_{m+n;t}),~\tilde B_{t}=\begin{pmatrix} B_{t}\\ C_{t}\end{pmatrix}$,  where $B_t=(b_{ij}^t),~C_t=(c_{ij}^t),~x_{n+i;t}=y_i$ for $i=1,2,\cdots,m.$
\end{Definition}

Let $\mathcal A_{t_0}$ be the cluster algebra with principal coefficients at $\Sigma_{t_0}$ (or say at $t_0$), the authors in \cite{FZ3} introduced a $\mathbb Z^n$-grading of $\mathbb Z[x_{1;t_0}^{\pm1},\cdots,x_{n;t_0}^{\pm1},y_1,\cdots,y_n]$ as follows:
$$deg(x_{i;t_0})={\bf e}_i,~deg(y_j)=-{\bf b}_j,$$
where ${\bf e}_i$ is the $i$-th column vector of $I_n$, and ${\bf b}_j$ is the $j$-th column vector of $B_{t_0}$, $i,j=1,2,\cdots,n$. As shown in \cite{FZ3} every cluster variable $x_{i;t}$ of $\mathcal A_{t_0}$ is homogeneous with respect to this $\mathbb Z^n$-grading. The {\bf $g$-vector} of a cluster variable $x_{i;t}$ is defined to be its degree with respect to the $\mathbb Z^n$-grading and we write $deg(x_{i;t})=(g_{1i}^t,~g_{2i}^t,~\cdots,~g_{ni}^t)^{\top}\in\mathbb Z^n$. Let $X_t$ be a cluster of $\mathcal A_{t_0}$. We call the matrix $G_t=(deg(x_{1;t}),\cdots,deg(x_{n;t}))$ the {\bf $G$-matrix} of the cluster $X_t$. And $C_t$ is called the {\bf $C$-matrix} of $\Sigma_t$, whose column vectors are called {\bf $c$-vectors} (see \cite{FZ3}).

\begin{Proposition}\label{cgmutation}
Let $\mathcal A_{t_0}$ be a cluster algebra with principal coefficients at $t_0$,~$\Sigma_t$ be a seed of $\mathcal A_{t_0}$ and $\Sigma_{t^{\prime}}=\mu_{k}(\Sigma)$, then

(i)~$c_{ij}^{t^\prime}=\begin{cases}-c_{ij}^t~,&  j=k;\\  c_{ij}^t+b_{ik}^t[-b_{kj}^t]_++[c_{ik}^t]_+b_{kj}^t~,&otherwise.\end{cases}$

(ii)~(Proposition 6.6, \cite{FZ3}) $g_{ij}^{t^\prime}=\begin{cases}g_{ij}^t~,&  j\neq k;\\  -g_{ik}^t+\sum_{l=1}^n g_{il}^t[b_{lk}^t]_+-\sum_{l=1}^n-b_{il}^{t_0}[c_{lk}^t],&j=k.\end{cases}$
\end{Proposition}
\begin{proof}
(i) It can be obtained from the equality (\ref{bmutation}).
\end{proof}

Let $\mathcal A_{t_0}$ be the cluster algebra with principal coefficients at $t_0$, each cluster variable $x_{i;t}$ can be expressed as a rational function $X_{i;t}^{t_0}\in\mathbb Q(x_{1;t_0},\cdots,x_{n;t_0},y_1,\cdots,y_n)$. $X_{i;t}^{t_0}$ is called the {\bf$X$-function} of cluster variable $x_{i;t}$.

\begin{Definition}
The {\bf $F$-polynomial} of $x_{i;t}$ is defined by $F_{i;t}^{t_0}=X_{i;t}^{t_0}|_{x_{1;t_0}=\cdots x_{n;t_0}=1}\in\mathbb Z[y_1,\cdots,y_m]$.
\end{Definition}

A vector ${\bf c}\in\mathbb Z^n$ is called {\bf sign-coherent} (\cite{FZ3}) if any two nonzero entries of {\bf c} have the same sign.
\begin{Proposition}\label{gf0} ~(\cite{FZ3,GHKK})
 Let $\mathcal A_{t_0}$ be a cluster algebra with principal coefficients at $t_0$, $x_{i;t}$ be a cluster variable of $\mathcal A_{t_0}$, then

(i) ~$x_{i;t}=\left(\prod\limits_{j=1}^nx_{j;t_0}^{g_{ji}^t}\right)F_{i;t}^{t_0}|_{\mathcal F}(\hat y_{1;t_0},\cdots,\hat y_{n;t_0})$, where $\hat y_{k;t_0}=y_k\prod\limits_{i=1}^n x_{i;t_0}^{b_{ik}^{t_0}}$ for $k\in\{1,2,\cdots,n\}$.

(ii)~  The following are equivalent:

(a)~ The column vectors of $C_t$ are sign-coherent;

(b)~ Each {\bf $F$-polynomial} has constant term 1;

(c)~ Each {\bf $F$-polynomial} has a unique monomial of maximal degree, which has coefficient $1$ and is divisible by all of other occurring monomials.
\end{Proposition}

\section{Affirmation of Conjecture \ref{coj1} in skew-symmetrizable case}

 In this section, we give an affirmation to Conjecture \ref{coj1} (see Theorem \ref{mainthm}) in the skew-symmetrizable case depending on the results in \cite{GHKK}, that is, Theorem \ref{poconj} and Proposition \ref{gf} below.

\begin{Theorem}\label{poconj}
(Positivity of Laurent Phenomenon, \cite{GHKK})
Any cluster variable $x_{i;t}$ can be expressed as a Laurent polynomial in $\mathbb Z_{\geq0}[y_1,\cdots,y_m][x_{1;t_0}^{\pm1},\cdots,x_{n;t_0}^{\pm1}]$ in a skew-symmetrizable
cluster algebra $\mathcal A(\Sigma_{t_0})$ with initial seed $\Sigma_{t_0}$.
\end{Theorem}

\begin{Proposition}\label{gf}
(\cite{GHKK}) Let $\mathcal A_{t_0}$ be a skew-symmetrizable cluster algebra with principal coefficients at $t_0$, $x_{i;t}$ be a cluster variable of $\mathcal A_{t_0}$, then
 the column vectors of $C_t$ are sign-coherent.
\end{Proposition}

\begin{Proposition}\label{cg}(\cite{CL,NZ,NT})
Let $\mathcal A_{t_0}$ be a skew-symmetrizable cluster algebra with principal coefficients $t_0$ and with skew-symmetrizer $S$, ~$\Sigma_{t}$ be a seed of $\mathcal A_{t_0}$, then $G_tB_tS^{-1}G_t^{\top}=B_{t_0}S^{-1}$ and $SC_tS^{-1}G_t^{\top}=I_n$ and $det(G_t)=\pm1$.
\end{Proposition}

\begin{Lemma}\label{mainlem}
Let $\mathcal A_{t_0}$ be a cluster algebra with principal coefficients at $t_0$, $\Sigma_t=(X_t,\tilde B_t)$ be a seed of $\mathcal A_{t_0}$. If $G_t=I_n$, then $\Sigma_t=\Sigma_{t_0}$.
\end{Lemma}
\begin{proof}
By Proposition \ref{cg} and $G_t=I_n=G_{t_0}$, we obtain  $B_t=B_{t_0}$ and $C_t=I_n$.
So we can also view $\mathcal A_{t_0}$ as a cluster algebra with principal coefficients at $t$.
Let $\Sigma_{t_1}$ be a seed of $\mathcal A_{t_0}$, the {\bf $C$}-matrix of $\Sigma_{t_1}$ is always the lower part of $\tilde B_{t_1}$,  no matter which seed ($\Sigma_{t_0}$ or $\Sigma_t$) is chosen as initial seed. By Proposition \ref{cg} again, we know the {\bf $G$}-matrices of $\Sigma_{t_1}$ with respect to $\Sigma_{t_0}$ and $\Sigma_t$ are the same.
By Proposition \ref{gf0} (i), we know
\begin{eqnarray}\label{eqt}
x_{i;t}=x_{i;t_0}F_{i;t}^{t_0}(y_1\prod\limits_{l=1}^n x_{l;t_0}^{b_{l1}^{t_0}},\cdots,y_n\prod\limits_{l=1}^n x_{l;t_0}^{b_{ln}^{t_0}}),
 \end{eqnarray}
 by viewing $\Sigma_{t_0}$ as an initial seed and
  \begin{eqnarray}\label{eqt0}
  x_{j;t_0}=x_{j;t}F_{j;t_0}^{t}(y_1\prod\limits_{l=1}^n x_{l;t}^{b_{l1}^{t}},\cdots,y_n\prod\limits_{l=1}^n x_{l;t}^{b_{ln}^{t}}),
   \end{eqnarray}
  by viewing $\Sigma_{t}$ as an initial seed for $i,j\in\{1,2,\cdots,n\}$.
  Replacing $x_{1;t_0},\cdots,x_{n;t_0}$ in (\ref{eqt}) with the ones in (\ref{eqt0}), we obtain

$ x_{i;t}=x_{i;t}F_{i;t_0}^{t}(y_1\prod\limits_{l=1}^n x_{l;t}^{b_{l1}^{t}},\cdots,y_n\prod\limits_{l=1}^n x_{l;t}^{b_{ln}^{t}})\cdot\nonumber\\
 F_{i;t}^{t_0}\left(y_1\prod\limits_{l=1}^n (x_{l;t}F_{l;t_0}^{t}(y_1\prod\limits_{k=1}^n x_{k;t}^{b_{k1}^{t}},\cdots,y_n\prod\limits_{k=1}^n x_{k;t}^{b_{kn}^{t}}))^{b_{l1}^{t_0}},\cdots,y_n\prod\limits_{l=1}^n (x_{l;t}F_{l;t_0}^{t}(y_1\prod\limits_{k=1}^n x_{k;t}^{b_{k1}^{t}},\cdots,y_n\prod\limits_{k=1}^n x_{k;t}^{b_{kn}^{t}}))^{b_{ln}^{t_0}}\right)\nonumber,$
which implies that

$ 1=F_{i;t_0}^{t}(y_1\prod\limits_{l=1}^n x_{l;t}^{b_{l1}^{t}},\cdots,y_n\prod\limits_{l=1}^n x_{l;t}^{b_{ln}^{t}})\cdot\nonumber\\
 F_{i;t}^{t_0}\left(y_1\prod\limits_{l=1}^n (x_{l;t}F_{l;t_0}^{t}(y_1\prod\limits_{k=1}^n x_{k;t}^{b_{k1}^{t}},\cdots,y_n\prod\limits_{k=1}^n x_{k;t}^{b_{kn}^{t}}))^{b_{l1}^{t_0}},\cdots,y_n\prod\limits_{l=1}^n (x_{l;t}F_{l;t_0}^{t}(y_1\prod\limits_{k=1}^n x_{k;t}^{b_{k1}^{t}},\cdots,y_n\prod\limits_{k=1}^n x_{k;t}^{b_{kn}^{t}}))^{b_{ln}^{t_0}}\right)\nonumber,$

Take $x_{1;t}=x_{2;t}=\cdots=x_{n;t}=1$ in the above equality, then we have
  \begin{eqnarray}\label{yequation}
&&1=F_{i;t_0}^{t}(y_1,\cdots,y_n)\cdot
 F_{i;t}^{t_0}\left(y_1\prod\limits_{l=1}^n (F_{l;t_0}^{t}(y_1,\cdots,y_n))^{b_{l1}^{t_0}},\cdots,y_n\prod\limits_{l=1}^n (F_{l;t_0}^{t}(y_1,\cdots,y_n)^{b_{ln}^{t_0}} \right),
  \end{eqnarray}
  By Theorem \ref{poconj} and the definition of {\bf $F$}-polynomial,  we have $F_{j;t}^{t_0},~F_{j;t_0}^{t}\in\mathbb Z_{\geq0}[y_1,\cdots,y_n]$. Then
  we know each $y_j\prod\limits_{l=1}^n (F_{j;t_0}^{t}(y_1,\cdots,y_n))^{b_{lj}^{t_0}}$ has the form of $\frac{h_j}{g_j}$, where $h_j,g_j\in\mathbb Z_{\geq0}[y_1,\cdots,y_n],~j=1,\cdots,n$.
  We claim that $F_{i;t}^{t_0}(y_1,\cdots,y_n)=1$. If $F_{i;t}^{t_0}(y_1,\cdots,y_n)\neq1$, then
  by Proposition \ref{gf} and Proposition \ref{gf0} (ii), we can assume $F_{i;t}^{t_0}=y_1^{k_{1i}}y_2^{k_{2i}}\cdots y_n^{k_{ni}}+u_i(y_1,\cdots,y_n)+1$, where $(k_{1i},\cdots,k_{ni})\neq0$, ~$u_i\in\mathbb Z_{\geq0}[y_1,\cdots,y_n]$ and each monomial in $u_i$ is a factor of $y_1^{k_{1i}}y_2^{k_{2i}}\cdots y_n^{k_{ni}}$. And we assume $F_{i;t_0}^t=w_i(y_1,\cdots,y_n)+1$, where $w_i\in\mathbb Z_{\geq0}[y_1,\cdots,y_n]$.

  Then the equality (\ref{yequation}) has the form of
  \begin{eqnarray}
  1=(w_i+1)\left((\frac{h_1}{g_1})^{k_{1i}}\cdots (\frac{h_n}{g_n})^{k_{ni}}+u_i(\frac{h_1}{g_1},\cdots,\frac{h_n}{g_n})+1\right),\nonumber
  \end{eqnarray}
  thus  $g_1^{k_{1i}}g_2^{k_{2i}}\cdots g_n^{k_{ni}}=(w_i+1)(h_1^{k_{1i}}\cdots h_n^{k_{ni}}+g_1^{k_{1i}}\cdots g_n^{k_{ni}}u_i(\frac{h_1}{g_1},\cdots,\frac{h_n}{g_n})+g_1^{k_{1i}}\cdots g_n^{k_{ni}}).$

  We obtain  that
  \begin{eqnarray}\label{linearly}
 0&=&w_i\left(h_1^{k_{1i}}\cdots h_n^{k_{ni}}+g_1^{k_{1i}}g_2^{k_{2i}}\cdots g_n^{k_{ni}}u_i(\frac{h_1}{g_1},\cdots,\frac{h_n}{g_n})+g_1^{k_{1i}}\cdots g_n^{k_{ni}}\right)+\nonumber\\
&&\left(h_1^{k_{1i}}\cdots h_n^{k_{ni}}+g_1^{k_{1i}}g_2^{k_{2i}}\cdots g_n^{k_{ni}}u_i(\frac{h_1}{g_1},\cdots,\frac{h_n}{g_n})\right).
  \end{eqnarray}

  Note  that $g_1^{k_{1i}}g_2^{k_{2i}}\cdots g_n^{k_{ni}}u_i(\frac{h_1}{g_1},\cdots,\frac{h_n}{g_n})\in\mathbb Z_{\geq0}[y_1,\cdots,y_n]$, since each monomial occurring in $u_i(y_1,\cdots,y_n)$ is a factor of $y_1^{k_{1i}}y_2^{k_{2i}}\cdots y_n^{k_{ni}}$.
  Then the equality (\ref{linearly}) is a contradiction, since each term in the equality (\ref{linearly}) is an element in $\mathbb Z_{\geq0}[y_1,\cdots,y_n]$ and $h_1^{k_{1i}}\cdots h_n^{k_{ni}}\neq0$ by  the fact that $h_j\neq0,~j=1,\cdots,n$.
  Thus we must have  $F_{i;t}^{t_0}(y_1,\cdots,y_n)=1$, which implies $x_{i;t}=x_{i;t_0}$ by the equality (\ref{eqt}). Then $\Sigma_t=\Sigma_{t_0}$.
\end{proof}

Now, we obtain the main result as follows.
\begin{Theorem}\label{mainthm}
Let $\mathcal A_{t_0}$ be the cluster algebra with principal coefficients at $t_0$, $\Sigma_{t_1}$ and $\Sigma_{t_2}$ be two seeds of $\mathcal A_{t_0}$. If $C_{t_1}=C_{t_2}$, then $\Sigma_{t_1}=\Sigma_{t_2}$.
\end{Theorem}
\begin{proof}
By Proposition \ref{cg} and due to $C_{t_1}=C_{t_2}$, we have $B_{t_1}=B_{t_2}$ and $G_{t_1}=G_{t_2}$.
Since $\Sigma_{t_0}$ can be obtained from $\Sigma_{t_1}$ by a sequence of mutations,  we write  $\Sigma_{t_0}=\mu_{s_k}\cdots\mu_{s_2}\mu_{s_1}(\Sigma_{t_1})$. Let $\Sigma_t=\mu_{s_k}\cdots\mu_{s_2}\mu_{s_1}(\Sigma_{t_2})$.
Since $(G_{t_1},C_{t_1},B_{t_1})=(G_{t_2},C_{t_2},B_{t_2})$, by Proposition \ref{cgmutation}, we know that $(G_{t},C_{t},B_{t})=(G_{t_0},C_{t_0},B_{t_0})=(I_n,I_n,B_{t_0})$. By Lemma \ref{mainlem}, we know $\Sigma_t=\Sigma_{t_0}$ and thus
$\mu_{s_1}\mu_{s_2}\cdots\mu_{s_k}(\Sigma_t)=\mu_{s_1}\mu_{s_2}\cdots\mu_{s_k}(\Sigma_{t_0})$, i.e. $\Sigma_{t_2}=\Sigma_{t_1}$.
\end{proof}

As readers can see that the proof  of this result depends mainly  on the  positivity of cluster variables and sign-coherence of $c$-vectors of a cluster algebra.  Positivity and sign-coherence are two  deep phenomena in cluster algebras. It is known in \cite{NZ} that sign-coherence can relate many other properties of cluster algebras.  Also, we believe that positivity can be used to explain some  other properties of cluster algebras.  Through the method of the proof of Theorem \ref{mainthm}, we attempt to provide an evidence for the essentiality of positivity and sign-coherence.

 It was proved  that $\Sigma_t$ is uniquely determined by $G_t$ for a skew-symmetric matrix $B_{t_0}$ in \cite{DWZ}. By Proposition \ref{cg}, $C_{t_1}=C_{t_2}$ if and only if $G_{t_1}=G_{t_2}$. So, the result of Theorem \ref{mainthm} means that $\Sigma_t$ is uniquely determined by $G_t$, which gives a generalization of the  result in \cite{DWZ} in the skew-symmetrizable case.

\section{How to determine seed in sign-skew-symmetric case}

 Finally, we discuss how to determine a seed using $C$-matrix or $G$-matrix in sign-skew-symmetric case. The obtained result may be thought as the weak version of Conjecture \ref{coj1}.

 In this section, we always assume that  $\mathcal A_{t_0}$ is an acyclic sign-skew-symmetric cluster algebra with principal coefficients at $t_0$. Here, the ``acyclic" condition is assumed since we need the following conclusion.

\begin{Proposition}(Theorem 7.11 and 7.13, \cite{HL})\label{fp}
Let $\mathcal A_{t_0}$ be an acyclic sign-skew-symmetric cluster algebra with principal coefficients at $t_0$, then
\\
(i) Each cluster variable $x_{i;t}\in \mathbb Z_{\geq0}[y_1,\cdots,y_m][x_{1;t_0}^{\pm1},\cdots,x_{n;t_0}^{\pm1}]$.
\\
(ii) Each {\bf F-polynomial} has constant term $1$.
\end{Proposition}

\begin{Lemma}\label{lemsign}
Let $\mathcal A_{t_0}$ be an acyclic sign-skew-symmetric cluster algebra with principal coefficients at $t_0$. If $\Sigma_t=(X_t,\tilde B_t)$ is a seed of $\mathcal A_{t_0}$ satisfying $(G_t, C_t, B_t)=(G_{t_0}, C_{t_0}, B_{t_0})$, then $\Sigma_t=\Sigma_{t_0}$.
\end{Lemma}
{\em Sketch of Proof:}\;
Since $C_t=C_{t_0}=I_n$ and $G_t=G_{t_0}=I_n$, $\mathcal A_{t_0}$ can be also seen as a cluster algebra with principal coefficients at $t$. Base on Proposition \ref{fp}, we can use the same method with that of Lemma \ref{mainlem} so as to show $X_t=X_{t_0}$, which implies $\Sigma_t=\Sigma_{t_0}$ since $\tilde B_t=\tilde B_{t_0}$ is already known.\;\;\;\;\;\; $\square$\\

Note that since we have not the skew-symmetrizer $S$ in this case, Proposition \ref{cg} can not be used here as like in the proof of Lemma \ref{mainlem}. This is the reason we need the condition $(G_t, C_t, B_t)=(G_{t_0}, C_{t_0}, B_{t_0})$.

Following Lemma \ref{lemsign}, using the same method with the proof of Theorem \ref{mainthm}, we have:

\begin{Proposition}\label{mainpro}
Let $\mathcal A_{t_0}$ be an acyclic sign-skew-symmetric cluster algebra with principal coefficients at $t_0$. If $\Sigma_{t_1}$ and $\Sigma_{t_1}$ are two seeds of $\mathcal A_{t_0}$ with $(G_{t_1},C_{t_1},B_{t_1})=(G_{t_2},C_{t_2},B_{t_2})$, then $\Sigma_{t_1}=\Sigma_{t_2}$.
\end{Proposition}

\begin{Remark}
When $B_{t_0}$ is skew-symmetrizable, by Proposition \ref{cg}, either $G_{t_1}=G_{t_2}$ or $C_{t_1}=C_{t_2}$  implies $(G_{t_1},C_{t_1},B_{t_1})=(G_{t_2},C_{t_2},B_{t_2})$. So, Proposition \ref{mainpro} can be thought as a weak version of Theorem \ref{mainthm} in the acyclic sign-skew-symmetric case.
\end{Remark}

{\bf Acknowledgements:}\; This project is supported by the National Natural Science Foundation of China (No.11671350 and No.11571173).


\end{document}